\documentclass[12pt]{amsart}
\usepackage{amsmath,amssymb,bm}
\usepackage{graphicx,color}
\usepackage{subfigure}
\newtheorem{theorem}{Theorem}[section]
\newtheorem{lemma}[theorem]{Lemma}
\newtheorem{remark}[theorem]{Remark}
\numberwithin{equation}{section}
\numberwithin{table}{section}
\numberwithin{figure}{section}
\allowdisplaybreaks[4]
\begin{document}
\def\O{\Omega}
\def\R{\mathbb{R}}
\def\p{\partial}
\def\argmin{\mathop{\rm argmin}}
\def\ES{H^2_0(\O)}
\def\cN{\mathcal{N}}
\def\cA{\mathcal{A}}
\def\fA{\mathfrak{A}}
\def\cI{\mathcal{I}}
\def\tTh{\tilde T_h}
\def\tcN{\widetilde\cN}
\def\tV{\widetilde V}
\def\tA{\tilde A}
\def\BOL{B_{\scriptscriptstyle\rm OL}}
\def\BTL{B_{\scriptscriptstyle\rm TL}}
\def\red{\color{red}}
\title[Additive Schwarz Preconditioners for the Obstacle Problem of Plates]
{Additive Schwarz Preconditioners for the Obstacle Problem of  Clamped Kirchhoff Plates}
\author{Susanne C. Brenner}
\address{Department of Mathematics and Center for Computation and Technology,
Louisiana State University, Baton Rouge, LA 70803}
\email{brenner@math.lsu.edu}
\author{Christopher B. Davis}
\address{Department of Mathematics, Tennessee Technological University,
 Cookeville TN 38505}
 \email{cbdavis@tntech.edu}
\author{Li-yeng Sung}
\address{Department of Mathematics and Center for Computation and Technology,
Louisiana State University, Baton Rouge, LA 70803}
\email{sung@math.lsu.edu}
\thanks{The work of the first and third authors was supported in part
 by the National Science Foundation under Grant No.
 DMS-16-20273. Portions of this research were conducted with high performance computing resources
 provided by Louisiana State University (http://www.hpc.lsu.edu)}
\begin{abstract}
 When the obstacle problem of clamped Kirchhoff plates is discretized by
 a partition of unity method, the resulting discrete variational inequalities can be solved by
 a primal-dual active set algorithm.   In this paper we
 develop and analyze additive Schwarz preconditioners for the systems that appear
 in each iteration of the  primal-dual active set algorithm.  Numerical results that corroborate
 the theoretical estimates are also presented.
\end{abstract}

\date{September 3, 2018}
\maketitle
\section{Introduction}\label{sec:Introduction}
 Let $\O$ be a bounded polygonal domain in $\R^2$, $f\in L_2(\O)$
 and $\psi\in C(\bar\O)\cap C^2(\O)$ such that
 $\psi<0$ on $\p\O$.
  The obstacle problem for a clamped
 Kirchhoff plate occupying $\O$ is to find
\begin{equation}\label{eq:OP}
 u=\argmin_{v\in K}\Big[\frac12 a(v,v)-(f,v)\big],
\end{equation}
 where $(\cdot,\cdot)$ is the inner product for $L_2(\O)$,
\begin{align}\label{eq:aDef}
 a(v,w)=\int_\O D^2 v: D^2 w\, dx
 =\int_\O \sum_{i,j=1}^2 \Big(\frac{\p^2 v}{\p x_i\p x _j}\Big)
    \Big(\frac{\p^2 w}{\p x_i\p x _j}\Big)dx\qquad\forall\,v,w\in \ES,
\end{align}
 and $K$ is the subset of $\ES$ defined by
\begin{equation}\label{eq:KDef}
 K=\{v\in\ES:\, v\geq \psi \;\text{on}\;\O\}.
\end{equation}
 Here and
 throughout the paper we follow the standard notation for differential operators, function spaces and
 norms that can be found for example in \cite{Ciarlet:1978:FEM,ADAMS:2003:Sobolev,BScott:2008:FEM}.
\par
 Since $K$ is a nonempty closed convex subset of the Hilbert space $\ES$,
 it follows from the standard theory of calculus of variations
  \cite{ET:1999:Convex,KS:1980:VarInequalities} that
 the obstacle problem \eqref{eq:OP}
  has a unique solution $u\in K$ characterized by the fourth order variational inequality
\begin{equation}\label{eq:VI}
  a(u,v-u)-(f,v-u)\geq 0 \qquad\forall\,v\in K.
\end{equation}
\par
 The numerical solution of the obstacle problem \eqref{eq:OP}--\eqref{eq:KDef} by a
 generalized finite element method was studied  in \cite{BDS:2014:GFEM}.  The discrete variational
 inequalities resulting from the generalized finite element method were solved by
 a primal-dual active set algorithm \cite{BIK:1999:PDAS,BK:2002:PDAS,HIK:2003:PDAS,IK:2008:Lagrange},
 where an auxiliary system of equations involving the inactive nodes had to be solved in each iteration.
 Since this is a fourth order problem, these systems become very ill-conditioned
  when the number of degrees of freedom becomes large.  The goal of this paper is to develop
  one-level and two-level additive Schwarz
 domain decomposition preconditioners for
 the systems that appear in the primal-dual active set algorithm.
 We note that a
 two-level additive Schwarz preconditioner for the plate bending problem (without the obstacle)
 using the same generalized finite element method  was investigated in \cite{BDS:2017:TLDDPUM}.
\par
 There is a sizable literature on
 domain decomposition methods for second order variational inequalities  \cite{KNT:1992:Obstacle,ZZ:1998:Obstacle,BW:2000:AS,Tai:2001:Obstacle,TT:2002:Convex,
 TX:2002:Convex,Tai:2003:MGVI,BTW:2003:MS,Mathew:2008:DD,
 CZ:2008:Obstacle,BK:2012:ADSVI,Lee:2013:DDVI}.  (References for related multigrid methods can be found
 in the survey article \cite{GK:2009:MGObstacle}.)
  On the other hand
 the literature on domain decomposition methods for fourth order variational inequalities is
 quite small.  The only work \cite{Scarpini:1990:Biharmonic} that we know of
  treats an alternating Schwarz algorithm for the plate
 obstacle problem discretized by a mixed finite element method.
\par
 We note that most of the domain decomposition algorithms for variational inequalities are based on the
 subspace correction approach except the one in \cite{Lee:2013:DDVI}, where the author considered a multibody
 second order elliptic problem with inequality constraints on the interfaces of the bodies, and
 the nonoverlapping domain decomposition preconditioners in that paper
 are also designed for the auxiliary systems that appear in a
 primal-dual active set algorithm.
\par
 The rest of the paper is organized as follows.  We recall the partition
 of unity method in Section~\ref{sec:PTUM} and the primal-dual active set algorithm
  in Section~\ref{sec:PDAS}.  We set up overlapping domain decomposition in Section~\ref{sec:DD}
  and study the one-level and two-level additive Schwarz preconditioners in Section~\ref{sec:OneLevel}
  and Section~\ref{sec:TwoLevel}.  Numerical results that corroborate the theoretical estimates
  are presented in Section~\ref{sec:Numerics} and we end with some concluding remarks in
   Section~\ref{sec:Conclusions}.
\section{A Partition of Unity Method}\label{sec:PTUM}
 Conforming finite element methods for the fourth order problem \eqref{eq:OP}--\eqref{eq:KDef}
  require $C^1$ finite element spaces.
 In the classical setting this would involve polynomials of high degrees in the construction of
 the local approximation spaces \cite{Ciarlet:1978:FEM,BScott:2008:FEM}.
 An alternative is to employ
 generalized finite element methods \cite{MB:1996:PTU,BBO:2003:GFEM}.
 This was carried out in \cite{BDS:2014:GFEM} using a flat-top partition of unity method (PUM)
 from \cite{GS:2002:PUM,OKH:2008:PTU,ODJ:2011:Plate,Davis:2014:PTU}.
 Below we recall some basic facts concerning the PUM in \cite{BDS:2014:GFEM}.
\par
 Let $\{\O_i\}_{i=1}^n$ be an open cover of $\bar\O$
 such that there exists a collection of nonnegative functions
  $\phi_1,\ldots,\phi_n\in W^2_\infty(\R^2)$
 with the following properties:
\begin{alignat*}{3}
 \text{supp}\,\phi_i&\subset \O_i &\qquad&\text{for}\;1\leq i\leq n,\\
 \sum_{i=1}^n\phi_i&=1   &\qquad&\text{on}\;\O,\\
 |\phi_i|_{W^m_\infty(\R^2)}&\leq \frac{C}{(\text{diam}\,\O_i)^m} &\qquad&\text{for}\;
    0\leq m\leq 2,\;1\leq i\leq n.
\end{alignat*}
 From here on we use $C$ (with or without subscript) to denote a generic positive constant
 that can take different values at different appearances.
\par
 Let $V_i$ be a subspace of biquadratic polynomials defined on the local patch $\O_i$ whose members
 satisfy the homogeneous Dirichlet boundary conditions on $\p\O$.  The generalized finite element space
 $V_G\subset\ES$ is given by
\begin{equation*}
  V_G=\sum_{i=1}^n \phi_iV_i.
\end{equation*}
\par
 There are many choices in the construction of the partition of unity.  We use a  flat-top partition of
 unity where each $\O_i$ is an open rectangle and each $\phi_i$ is the tensor product of two
 one dimensional flat-top functions.   The flat-top region $\O_i^{\rm flat}$ inside $\O_i$ is the set where
 $\phi_i=1$, and the degrees of freedom for the local space $V_i$ are all associated with nodes on
 $\O_i^{\rm flat}$. An illustration for such a construction  is given in Figure~\ref{Fig:PTU} for a
 square domain $\O$.  Details for the construction and examples for other domains can be found
  in \cite{BDS:2014:GFEM}.
\begin{figure}[!htb]
\centering
\subfigure[]{
   \includegraphics[width=0.25\textwidth]{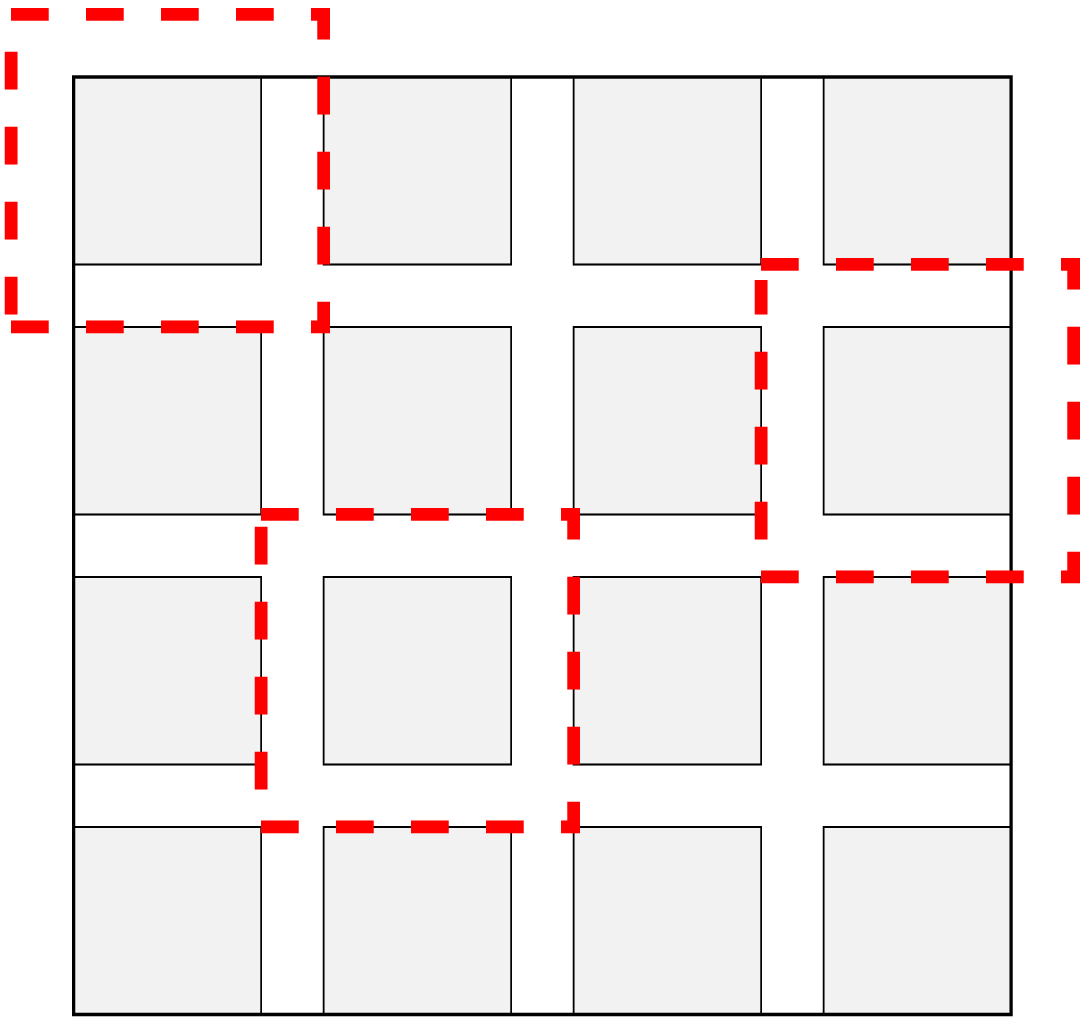}}
\hspace{30pt}
\subfigure[]{
   \includegraphics[width=0.25\textwidth]{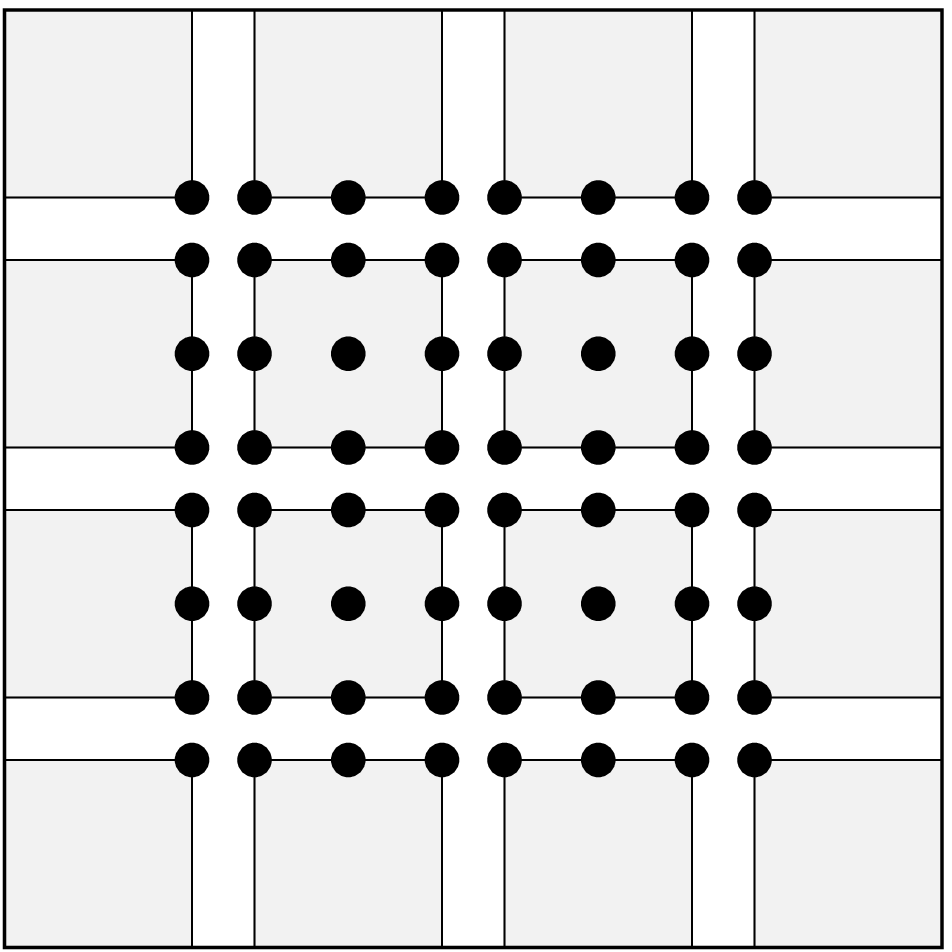}}
\caption{(a) $\O_i$ (bounded by dotted lines) and $\O_i^{\rm flat}$ (shaded in grey)
(b) nodes for the interior DOFs}
\label{Fig:PTU}
\end{figure}
\par
 From now on we assume that the diameters of the patches are
  comparable to a mesh size $h$ and denote the
 generalized finite element space by $V_h$.
 Let $\cN_h$ be the set of the nodes
  in the local patches (solid dots in Figure~\ref{Fig:PTU}~(b)) that correspond to the degrees of freedom
  of the local basis functions.  (The cardinality of $\cN_h$ is the dimension of $V_h$.)
 The discrete problem is to find
\begin{equation}\label{eq:DP}
  u_h=\argmin_{v\in K_h}\Big[\frac12 a(v,v)-(f,v)\Big],
\end{equation}
 where
\begin{equation}\label{eq:KhDef}
 K_h=\{v\in V_h:\,v(p)\geq \psi(p) \quad\forall\,p\in \cN_h\}.
\end{equation}
\begin{remark}\label{rem:Box}\rm
 Since the nodes in $\cN_h$ are located at the flat-top regions of the local patches, the constraints for $K_h$ are
  box constraints.
\end{remark}
\par
 Let the interpolation operator $\Pi_h:\ES\longrightarrow V_h$ be defined by
\begin{equation}\label{eq:Pih}
  \Pi_h \zeta=\sum_{i=1}^n (\Pi_i \zeta)\phi_i,
\end{equation}
 where $\Pi_i$ is the local nodal interpolation operator for $V_i$.
 Then $\Pi_h u$ belongs to $K_h$ and hence
 $K_h$ is a nonempty closed convex subset of $V_h$.
 It follows from the standard theory that
 \eqref{eq:DP} has a unique solution $u_h\in K_h$ characterized by the discrete variational
 inequality
\begin{equation}\label{eq:DVI}
 a(u_h,v-u_h)-(f,v-u_h)\geq0 \qquad\forall\,v\in K_h.
\end{equation}
 Moreover we have \cite[Theorem~3.2]{BDS:2014:GFEM}
\begin{equation*}
  |u-u_h|_{H^2(\O)}\leq Ch^\alpha,
\end{equation*}
 where the index of elliptic regularity $\alpha\in (\frac12,1]$ is determined by the angles at the corners
 of $\O$ and we can take $\alpha$ to be $1$ if $\O$ is convex.
\par
 We will need the  following interpolation error estimate
  \cite[(2.9)]{BDS:2014:GFEM} in the analysis of the domain decomposition
 preconditioners:
\begin{equation}\label{eq:InterpolationErrors}
  \sum_{m=0}^1 h^m|\zeta-\Pi_h\zeta|_{H^m(\O)}
  +h^2|\Pi_h\zeta|_{H^2(\O)}\leq C_{\Pi} h^2|\zeta|_{H^2(\O)} \qquad
  \forall\,\zeta\in H^2(\O),
\end{equation}
 where the positive constant $C_\Pi$ is independent of $h$.
\par
 We will also need the trivial estimate
\begin{equation}\label{eq:LTwo}
  \|v\|_{L_2(\O)}\approx \sum_{p\in\cN_h}v^2(p) \qquad\forall\,v\in V_h
\end{equation}
 that follows from standard estimates for the biquadratic polynomials defined over the patches.
%
\section{A Primal-Dual Active Set Algorithm}\label{sec:PDAS}
 Let the function $\lambda_h:\cN_h\longrightarrow\R$ be defined by
\begin{equation}\label{eq:lambdah}
  a(u_h,v)-(f,v)=\sum_{p\in\cN_h}\lambda_h(p)v(p) \qquad\forall\,v\in V_h.
\end{equation}
 The discrete variational inequality \eqref{eq:DVI} is equivalent to \eqref{eq:lambdah}
together with the optimality conditions
\begin{equation*}
  u_h(p)-\psi(p)\geq 0, \quad \lambda_h(p)\geq0 \quad\text{and}\quad
  (u_h(p)-\psi(p))\lambda_h(p)=0 \qquad \forall\,p\in\cN_h,
\end{equation*}
 which can also be written concisely  as
\begin{equation}\label{eq:Nonsmooth}
 \lambda_h(p)=\max\big(0,\lambda_h(p)+c(\psi(p)-u_h(p))\big)\qquad\forall\,p\in\cN_h.
\end{equation}
 Here $c$  can be any positive number.
\par
 The system defined by
 \eqref{eq:lambdah} and \eqref{eq:Nonsmooth} can be solved by a semi-smooth Newton iteration
 that is equivalent to a primal-dual active set method
 \cite{BIK:1999:PDAS,BK:2002:PDAS,HIK:2003:PDAS,IK:2008:Lagrange}.
 Given an approximation $(u_k,\lambda_k)$ of $(u_h,\lambda_h)$, the semi-smooth Newton iteration
 obtains the next approximation by solving the following system of equations:
\begin{subequations}\label{subseq:Newton}
\begin{alignat}{3}
  a(u_{k+1},v)-(f,v)&=\sum_{p\in\cN_h}\lambda_{k+1}(p)v(p) &\qquad&\forall\,v\in V_h,\label{eq:N1}\\
   u_{k+1}(p)&=\psi(p)&\qquad&\forall\,p\in\fA_k,\label{eq:N2} \\
   \lambda_{k+1}(p)&=0&\qquad&\forall\,p\in\cN_h\setminus\fA_k,\label{eq:N3}
\end{alignat}
 \end{subequations}
 where $\fA_k=\{p\in\cN_h: \lambda_k(p)+c(\psi(p)-u_k(p))>0\}$ is
 the active set determined by $(u_k,\lambda_k)$ and $c$ is a (large) positive constant.
   The  iteration terminates when $\fA_{k+1}=\fA_k$.
\par
   In view of  \eqref{eq:N2} and \eqref{eq:N3},  we can reduce \eqref{eq:N1} to an auxiliary system
    that only involves the
 unknowns $u_{k+1}(p)$ for $p\in \cN_h\setminus\fA_k$.  For small $h$, this is a large, sparse and
 ill-conditioned system that can be solved efficiently by a preconditioned Krylov subspace method,
 such as the  preconditioned conjugate gradient method.
\par
 This preconditioning problem can be posed in the following general form.
  Let $\tcN_h$ be a subset of $\cN_h$.  We define the truncation operator
  $\tTh:V_h\longrightarrow V_h$ by
\begin{equation}\label{eq:tThDef}
 (\tTh v)(p)=\begin{cases}
   v(p) &\qquad\text{if $p\in\tcN_h$}\\[4pt]
     0&\qquad\text{if $p\in\cN_h\setminus\tcN_h$}
 \end{cases}.
\end{equation}
 Then $\tTh$ is a projection from $V_h$ onto $\tV_h=\tTh V_h$.
 \par
 Let $\tA_h:\tV_h\longrightarrow \tV_h'$ be defined by
\begin{equation}\label{eq:tAhDef}
  \langle\tA_h v,w\rangle=a(v,w)\qquad\forall\,v,w\in \tV_h,
\end{equation}
 where $\langle\cdot,\cdot\rangle$ is the canonical bilinear form on $V_h'\times V_h$.
 We want to construct preconditioners for $\tA_h$
 whose performance is independent of $\tcN_h$.
  Since the partition of unity method is defined in terms of overlapping patches,
  it is natural to consider additive Schwarz domain decomposition preconditioners
 \cite{DW:1987:AS}.
\par
 Note that \eqref{eq:LTwo} implies
\begin{equation}\label{eq:TrucationEstimate}
  \|\tTh v\|_{L_2(\O)}\leq C\|v\|_{L_2(\O)} \qquad\forall\,v\in V_h.
\end{equation}
\section{Domain Decomposition}\label{sec:DD}
 Let the subdomains
 $\{D_j\}_{j=1}^J$ form an overlapping
 domain decomposition of $\O$ such that
\begin{equation}\label{eq:diamDj}
  \text{diam}\,D_j\approx H \qquad \text{for}\;1\leq j\leq J,
\end{equation}
 and
\begin{equation}\label{eq:Nc}
  \text{any point in $\O$ can belong to at most $N_c$ many subdomains}.
\end{equation}
 We also assume that the boundaries of $D_1,\ldots,D_J$ are aligned with the boundaries of the
 patches underlying the generalized finite element space $V_h$.  An example of four overlapping subdomains
 for a square domain $\O$ is depicted in Figure~\ref{Fig:DD}.  Details for the construction of
 $D_1,\ldots,D_J$ are available in \cite{BDS:2017:TLDDPUM}.
\begin{figure}[htb]
  \includegraphics[width=.22\textwidth]{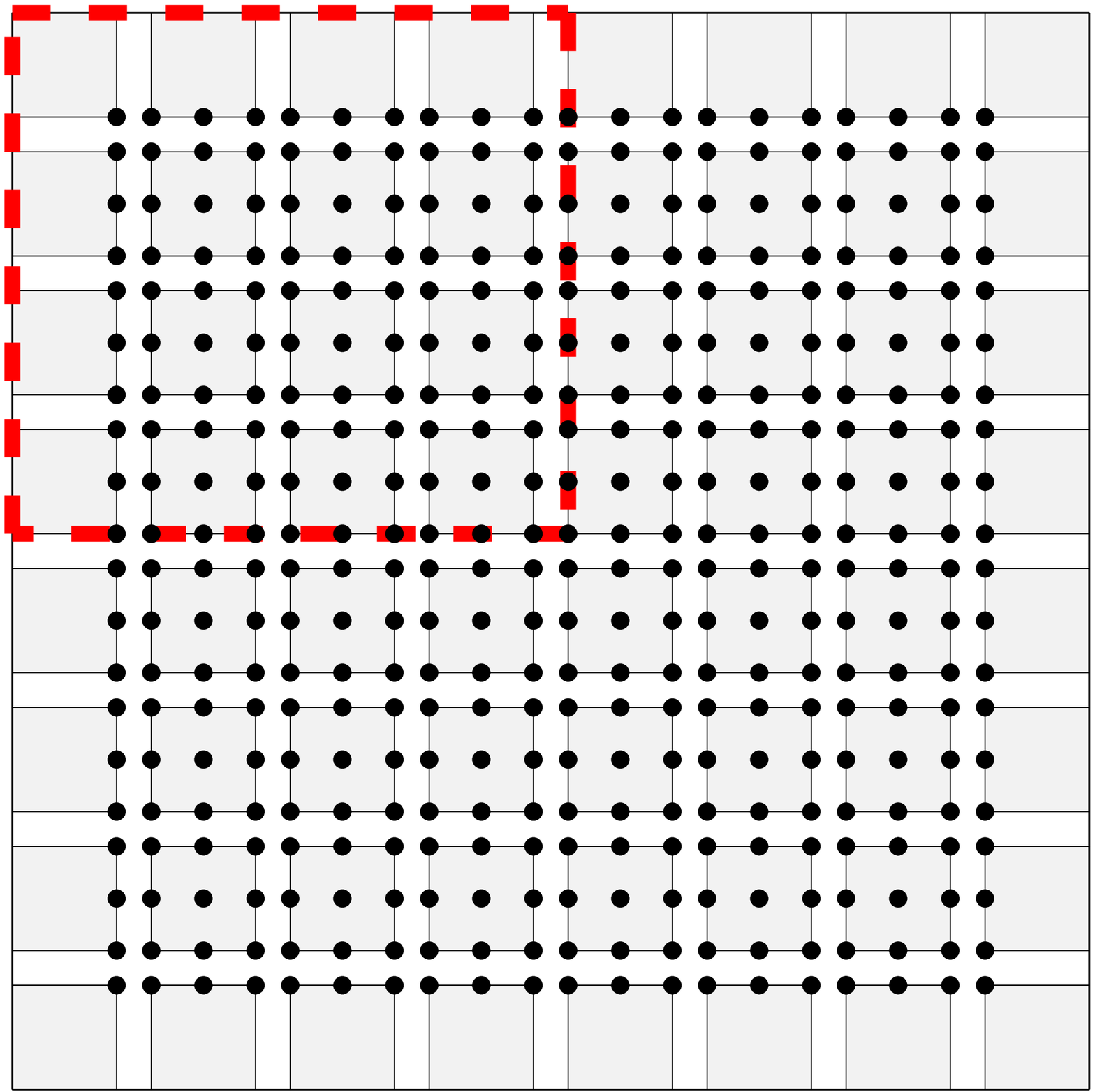}
  \hspace{8pt}
  \includegraphics[width=.22\textwidth]{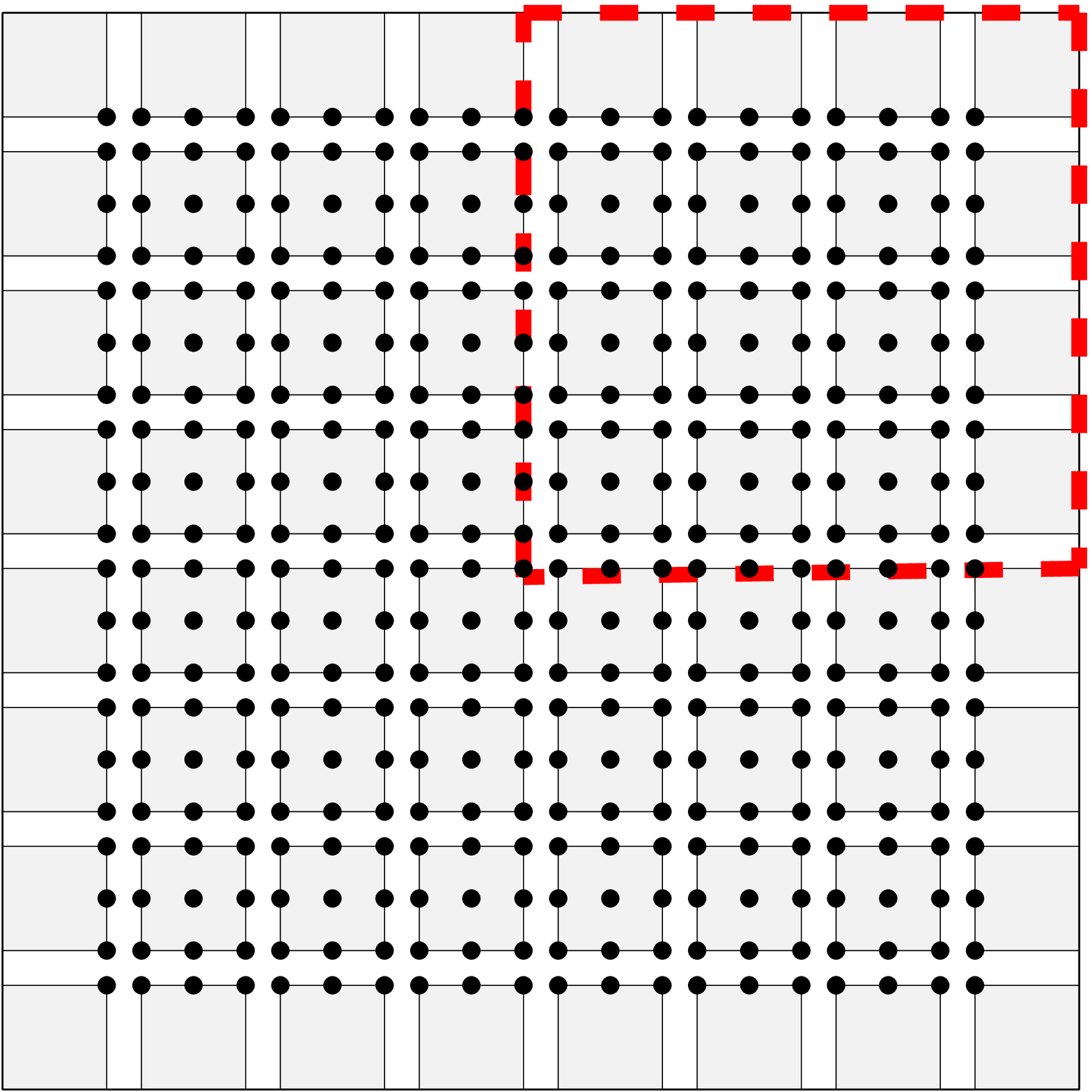}
   \hspace{8pt}
  \includegraphics[width=.22\textwidth]{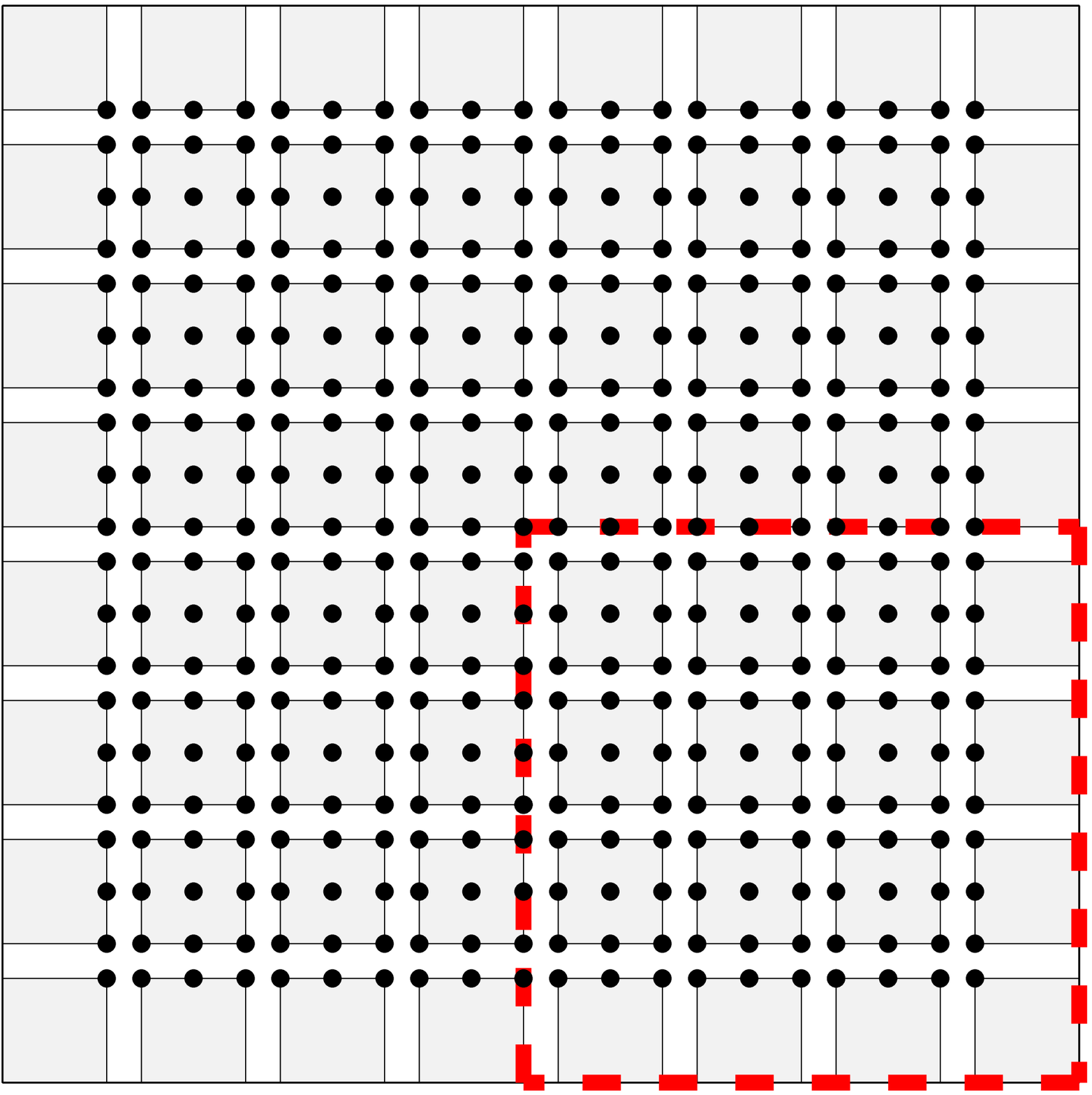}
   \hspace{8pt}
  \includegraphics[width=.22\textwidth]{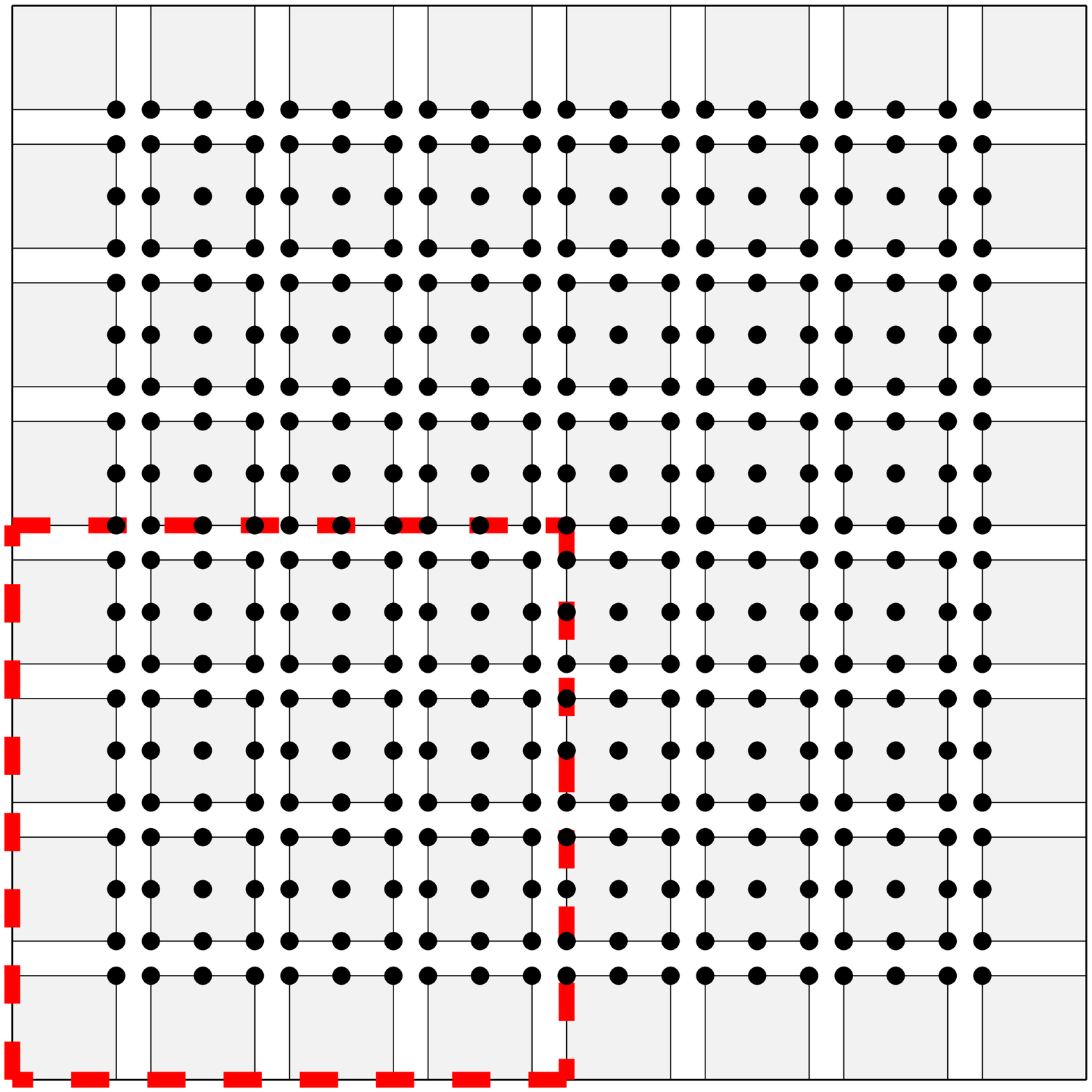}
\caption{Domain decomposition for a square domain $\O$
with four overlapping subdomains (bounded by the dotted lines)}
\label{Fig:DD}
\end{figure}
\par
 Note that \eqref{eq:diamDj} implies
\begin{equation}\label{eq:JandH}
  J\approx H^{-2},
\end{equation}
 provided that the subdomains $D_1,\ldots,D_J$ are shape regular.
\par
 We assume that there exists a partition of  unity
 $\psi_1,\ldots,\psi_J\in W^2_\infty(\R^2)$
  with the following properties:
\begin{subequations}\label{subeqs:PTU}
\begin{alignat}{3}
  \sum_{j=1}^J \psi_j&=1 &\qquad& \text{on} \; \bar\O,\\
  \psi_j&\geq 0 &\qquad &\text{for } j=1,\ldots J,\\
  \psi_j&=0 &\qquad&\text{on }\O\setminus\bar D_j,\; 1\leq j\leq J,\\
  |\psi_j|_{W^k_\infty(\O)}&\leq C_\dag \delta^{-k}&\qquad&
   \text{for } j=1,\ldots,J \text{ and } k=0,1,2.\label{eq:PTUBdd}
\end{alignat}
\end{subequations}
 Here $\delta$ $(\geq h)$ measures the overlap among the subdomains $D_1,\ldots,D_J$.

%
\section{A One-Level Additive Schwarz Preconditioner}\label{sec:OneLevel}
 Let $\tV_j$ be the subspace of $\tV_h$ whose members vanish at all the nodes
 outside $D_j$ and
 $A_j:\tV_j\longrightarrow \tV_j'$ be defined by
\begin{equation}\label{eq:AjDef}
 \langle A_jv,w\rangle=a(v,w) \qquad\forall\,v,w\in \tV_j.
\end{equation}
 The one-level additive Schwarz preconditioner $\BOL:\tV_h'\longrightarrow \tV_h$ is defined by
\begin{equation}\label{eq:BOL}
  \BOL=\sum_{j=1}^J I_j A_j^{-1}I_j^t,
\end{equation}
 where $I_j:\tV_j\longrightarrow \tV_h$ is the natural injection.
 \begin{theorem}\label{thm:OneLevel}
 We have
\begin{equation}\label{eq:OneLevelEstimate}
 \kappa(\BOL\tA_h)=\frac{\lambda_{\max}(\BOL\tA_h)}{\lambda_{\min}(\BOL\tA_h)}
 \leq C\delta^{-4},
\end{equation}
 where $\delta$ $(\geq h)$ measures the overlap among the subdomains $D_1,\ldots, D_J$ and the positive
 constant $C$ is independent of $h$, $H$, $J$, $\delta$ and $\tcN_h$.
\end{theorem}
\begin{proof} Let $v_j\in \tV_j$, $1\leq j\leq J$,  be arbitrary.  We have a standard estimate
\cite[Lemma~1]{BDS:2017:TLDDPUM}:
\begin{equation}\label{eq:OL1}
  a\Big(\sum_{j=1}^J I_jv_j,\sum_{j=1}^J I_jv_j\Big)\leq C_\sharp \sum_{j=1}^J a(v_j,v_j),
\end{equation}
 where the positive constant $C_\sharp$ only depends on $N_c$.
   It follows from the standard additive Schwarz theory
 \cite{SBG:1996:DD,TW:2005:DD,Mathew:2008:DD,BScott:2008:FEM} that
\begin{equation}\label{eq:OL2}
  \lambda_{\max}(\BOL\tA_h)\leq C_\sharp.
\end{equation}
\par
 Given any $v\in \tV_h$, we have
 $v_j=\Pi_h(\psi_j v)\in \tV_j$,
\begin{equation}\label{eq:OL3}
  \sum_{j=1}^J v_j=\Pi_h\Big[\Big(\sum_{j=1}^J \psi_j\Big)v\Big]=\Pi_h v=v,
\end{equation}
 and
\begin{align}\label{eq:OL4}
  \sum_{j=1}^J a(v_j,v_j)&=\sum_{j=1}^J a\big(\Pi_h(\psi_j v),\Pi_h(\psi_j v)\big)\notag\\
            &=\sum_{j=1}^J |\Pi_h(\psi_j v)|_{H^2(\O)}^2\notag\\
            &\leq  C\sum_{j=1}^J |\psi_j v|_{H^2(\O)}^2\notag\\
            &\leq C\sum_{j=1}^J\sum_{k=0}^2 |\psi_j|^2_{W^{2-k}_\infty(\O)}|v|_{H^k(D_j)}^2\\
            &\leq C\sum_{k=0}^2 \delta^{-2(2-k)}\sum_{j=1}^J |v|_{H^k(D_j)}^2\notag\\
            &\leq C\sum_{k=0}^2 \delta^{-2(2-k)}|v|_{H^k(\O)}^2\notag\\
           & \leq C_\flat \delta^{-4} |v|_{H^2(\O)}^2\notag\\
           &= C_\flat \delta^{-4} a(v,v),\notag
\end{align}
 where the positive constant $C_\flat$ depends only on
  $C_\Pi$ in \eqref{eq:InterpolationErrors},
  $N_c$ in \eqref{eq:Nc},
   $C_\dag$
 in \eqref{eq:PTUBdd} and constants for Poincar\'e-Friedrichs inequalities associated with $\ES$.
\par
 It follows from \eqref{eq:OL3}, \eqref{eq:OL4} and the standard additive Schwarz theory that
\begin{equation}\label{eq:OL5}
  \lambda_{\min}(\BOL\tA_h)\geq \delta^4 C_\flat^{-1}.
\end{equation}
\par
 The estimates \eqref{eq:OL2} and \eqref{eq:OL5} imply
 \eqref{eq:OneLevelEstimate} with $C=C_\sharp C_\flat$.
\end{proof}
\begin{remark}\label{rem:Same}\rm
  The estimate \eqref{eq:OneLevelEstimate} is identical to the one for the plate bending problem
  without an obstacle, i.e., the obstacle is invisible to the one-level additive Schwarz preconditioner.
\end{remark}
\begin{remark}\label{rem:Improvement1}\rm
 Under the assumption that the subdomains $D_1,\ldots,D_J$ are shape regular, we can
 improve the estimate \eqref{eq:OneLevelEstimate} to
\begin{equation}\label{eq:Improvement1}
 \kappa(\BOL\tA_h)\leq C\delta^{-3}H^{-1}
\end{equation}
 by the arguments in \cite[Section~8]{Brenner:1996:Plate}.  (Similar arguments for second order problems
 can be found in \cite[Lemma~3.10]{TW:2005:DD}.)  We will assume this is the case in
 the discussion below.
\end{remark}
\par
 Since $\delta$ decreases as $H$ decreases (or equivalently as $J$ increases), the one-level
 algorithm is not  a scalable algorithm.  Nevertheless the condition number estimate
 \eqref{eq:Improvement1} can still be a big improvement over the estimate
 $\kappa(\tA_h)\approx h^{-4}$ for the original system.
\subsection{Small Overlap}
 In the case of small overlap among the subdomains, we have $\delta\approx h$, and hence
\begin{equation}\label{eq:OLSmall1}
  \kappa(\BOL\tA_h)\leq Ch^{-3}H^{-1},
\end{equation}
 which indicates that asymptotically
\begin{align}
   &\text{$\kappa(\BOL\tA_h)$ will increase by a factor of $8$
 after each refinement if $H$ is kept fixed.} \label{eq:OLSmall2}\\
 \intertext{On the other hand the estimate \eqref{eq:OLSmall1} also indicates that}
  &\text{$\kappa(\BOL\tA_h)$ will increase by a factor of $2$ if $H$ decreases
    by a factor of 2}\label{eq:OLSmall3}\\
   &\text{ (or equivalently if $J$  increases by a factor of $4$) while $h$ is kept fixed.}\notag
\end{align}
%
\subsection{Generous Overlap}
 In the case of generous overlap among the subdomains, we have $\delta\approx H$, and hence
\begin{equation}\label{eq:OLLarge1}
  \kappa(\BOL\tA_h)\leq CH^{-4}\approx J^2.
\end{equation}
 It follows from  \eqref{eq:OLLarge1} that
\begin{align}
 &\text{$\kappa(\BOL\tA_h)$ increases as $J$ increases,}\label{eq:OLLarge2}\\
 \intertext{and}
  &\text{$\kappa(\BOL\tA_h)$ remains constant as $h$ decreases provided $H$
  (equivalently $J$)}\label{eq:OLLarge3}\\
    &\text{ is kept fixed.}\notag
 \end{align}
%
\section{A Two-Level Additive Schwarz Preconditioner}\label{sec:TwoLevel}
 Let $V_H$ be a coarse generalized finite element subspace of $\ES$
 associated with patches whose diameters are comparable to the diameters of the subdomains
 $D_1,\ldots,D_J$ in the decomposition of $\O$.
We define
  $\tV_0\subset \tV_h$ by
\begin{equation}\label{eq:W0}
 \tV_0=\tTh \Pi_hV_H,
\end{equation}
 and the operator $A_0:\tV_0\longrightarrow \tV_0'$ by
\begin{equation}\label{eq:A0Def}
 \langle A_0v,w\rangle=a(v,w) \qquad\forall\,v,w\in \tV_0.
\end{equation}
\par
 The two-level additive Schwarz preconditioner $\BTL:\tV_h'\longrightarrow \tV_h$ is given by
\begin{equation}\label{eq:BTL}
 \BTL=\sum_{j=0}^J I_j A_j^{-1} I_j^t,
\end{equation}
 where $I_0:\tV_0\longrightarrow \tV_h$ is also the natural injection.
\par
 Let $Q_H$ be the orthogonal projection from $L_2(\O)$ onto $V_H$.  The operator
 $R_0:V_h\longrightarrow\tV_0$ is defined by
\begin{equation}\label{eq:R0}
  R_0 v=\tTh\Pi_hQ_H v \qquad\forall\,v\in V_h.
\end{equation}
 The following result is useful for the analysis of $\BTL$.
\begin{lemma}\label{lem:R0}
  We have
\begin{equation}\label{eq:R0Estimates}
  \|v-R_0 v\|_{L_2(\O)}+h|v-R_0v|_{H^1(\O)}+
    h^2|v-R_0 v|_{H^2(\O)}\leq C H^2 |v|_{H^2(\O)} \qquad
  \forall\,v\in \tV_h.
\end{equation}
\end{lemma}
\begin{proof}  From \eqref{eq:InterpolationErrors} we have the estimate
\begin{equation*}
  \|\zeta-Q_H \zeta\|_{L_2(\O)}\leq CH^2|\zeta|_{H^2(\O)} \qquad\forall\,\zeta\in H^2(\O),
\end{equation*}
 which together with \eqref{eq:InterpolationErrors} and \eqref{eq:TrucationEstimate} implies
 that,  for any $v\in\tV_h$,
\begin{align*}
  \|v-R_0 v\|_{L_2(\O)}&=\|v-\tTh \Pi_hQ_H v\|_{L_2(\O)}\\
         &\leq \|\tTh(v-\Pi_hQ_H v)\|_{L_2(\O)}\\
         &\leq C\|v-\Pi_hQ_H v\|_{L_2(\O)}\\
         &\leq C\big(\|v-\Pi_hv\|_{L_2(\O)}+\|\Pi_h(v-Q_Hv)\|_{L_2(\O)}\big)\\
         &\leq C\big(\|v-\Pi_hv\|_{L_2(\O)}+\|v-Q_Hv\|_{L_2(\O)}\big)\\
         &\leq CH^2|v|_{H^2(\O)}.
\end{align*}
 The estimates for $|v-R_0v|_{H^1(\O)}$ and $|v-R_0v|_{H^2(\O)}$ then follow from inverse estimates.
\end{proof}
\begin{theorem}\label{thm:TwoLevel}
 We have
\begin{equation}\label{eq:TwoLeveLEstimate}
 \kappa(\BTL\tA_h)=\frac{\lambda_{\max}(\BTL\tA_h)}{\lambda_{\min}(\BTL\tA_h)}
 \leq C\min\big((H/h)^4,\delta^{-4}\big),
\end{equation}
 where the positive constant $C$ is independent of $H$, $h$, $J$, $\delta$ and $\tcN_h$.
\end{theorem}
\begin{proof}  The following upper bound for the maximum eigenvalue of $\BTL\tA_h$ is again standard
\cite[Lemma~1]{BDS:2017:TLDDPUM}:
\begin{equation}\label{eq:TL1}
  \lambda_{\max}(\BTL\tA_h)\leq \tilde C_\sharp,
\end{equation}
 where $\tilde C_\sharp$ only depends on the number $N_c$ in \eqref{eq:Nc}.
\par
 Let $v\in\tV_h$ be arbitrary, $v_0=R_0v\in \tV_0$ and $v_j=\Pi_h(\psi_j(v-v_0))\in \tV_j$.  We have
\begin{equation}\label{eq:TL2}
  \sum_{j=0}^J v_j=v_0+\Pi_h\Big[\Big(\sum_{j=1}^J\psi_j\Big)(v-v_0)\Big]=v_0+(v-v_0)=v,
\end{equation}
 and,  by \eqref{eq:R0Estimates},
\begin{align}
  a(v_0,v_0)=|R_0v|_{H^2(\O)}^2
            &\leq 2|v-R_0v|_{H^2(\O)}^2+2|v|_{H^2(\O)}^2\label{eq:TL3}\\
            &\leq C(1+H^4 h^{-4})|v|_{H^2(\O)}^2\leq
             C H^4 h^{-4}a(v,v).\notag
\end{align}
 Using \eqref{eq:PTUBdd} and \eqref{eq:R0Estimates}
 we also find
\begin{align} \label{eq:TL4}
  \sum_{j=1}^J a(v_j,v_j)&=\sum_{j=1}^J |\Pi_h(\psi_j(v-v_0))|_{H^2(\O)}^2\notag\\
     &\leq C \sum_{j=1}^J |\psi_j(v-v_0)|_{H^2(\O)}^2\notag\\
     &\leq C\sum_{j=1}^J \sum_{k=0}^2 |\psi_j|_{W^{2-k}_\infty(\O)}^2|v-R_0v|_{H^k(D_j)}^2
        \notag\\
        &\leq C\sum_{k=0}^2 \delta^{-2(2-k)}\sum_{j=1}^J |v-R_0v|_{H^k(D_j)}^2\\
        &\leq C\sum_{k=0}^2 \delta^{-2(2-k)}|v-R_0v|_{H^k(\O)}^2\notag\\
       &\leq C\Big(\frac{H^4}{\delta^4}+\frac{H^4}{\delta^2 h^2}
       +\frac{H^4}{h^4}\Big)|v|_{H^2(\O)}^2\notag\\
       &\leq C (H/h)^4a(v,v).
       \notag
\end{align}
\par
 It follows from \eqref{eq:TL2}--\eqref{eq:TL4} that
\begin{equation*}
  \sum_{j=0}^J a(v_j,v_j)\leq \tilde C_\flat (H/h)^4a(v,v).
\end{equation*}
 \par
 On the other hand, by taking $v_0=0$ and $v_j=\Pi_h(\psi_j v)$ for $1\leq j\leq J$, we have
\begin{equation*}
 \sum_{j=0}^J a(v_j,v_j) \leq C_\flat \delta^{-4} a(v,v)
\end{equation*}
 by \eqref{eq:OL4}.
 Hence the standard theory
  for additive Schwarz preconditioners implies that
\begin{equation}\label{eq:TL5}
 \lambda_{\min}(\BTL\tA_h)\geq \frac{1}{\min\big(\tilde C_\flat(H/h)^4,C_\flat\delta^{-4}\big)}.
\end{equation}
\par
 Consequently the estimate \eqref{eq:TwoLeveLEstimate} holds with
  $C=\tilde C_\sharp \max(\tilde C_\flat,C_\flat)$.
\end{proof}
\begin{remark}\label{rem:Different}\rm
 The estimate \eqref{eq:TwoLeveLEstimate} is different from the estimate for
 the plate bending problem
 without obstacles that reads
\begin{equation*}
 \kappa(\BTL A_h)\leq C\Big(\frac{H}{\delta}\Big)^4.
\end{equation*}
 This difference is caused by the necessity of truncation in the
  construction of $\tV_0$ when the obstacle
 is present.
\end{remark}
\begin{remark}\label{rem:Improvement2}\rm
  Under the assumption that the subdomains
  $D_1,\ldots,D_J$ are shape regular, the estimate \eqref{eq:TwoLeveLEstimate}
  can be improved to
\begin{equation}\label{eq:Improvement2}
  \kappa(\BTL\tA_h) \leq C\min\big((H/h)^4,\delta^{-3}H^{-1}\big).
\end{equation}
 We will assume  this is the case in the discussion below.
\end{remark}
\par
 The estimates \eqref{eq:Improvement2}
  indicate that the two-level algorithm is
  scalable as long as the ratio $H/h$ remains bounded, and
\begin{align}\label{eq:TwoBetter}
 &\text{the condition number for the
  two-level algorithm is (up to a constant) at least}\\
  &\text{ as good as the one-level algorithm.}\notag
\end{align}
 \subsection{Small Overlap}\label{subsec:TLSmall}
  In the case of small overlap where $\delta\approx h$, we have
\begin{equation*}
  (H/h)^4\ll h^{-3}H^{-1}\quad \text{if $H^5\ll h$},
 \end{equation*}
  which indicates that
\begin{equation}\label{eq:TLSmall1}
  \kappa(\BTL\tA_h)<\kappa(\BOL\tA_h) \quad\text{for small $H$ (or equivalently
  for large $J$) if $h$ is kept fixed.}
\end{equation}
 \subsection{Generous Overlap}\label{subsec:TLLarge}
 In the case of generous overlap where $\delta\approx H$, we have the following
 analog of \eqref{eq:OLLarge3}:
\begin{align}\label{eq:TLLarge1}
  &\text{$\kappa(\BTL\tA_h)$ remains constant as $h$ decreases provided $H$
  (equivalently $J$)}\\
    &\text{ is kept fixed.}\notag
\end{align}
\par
 Moreover, we have
\begin{equation*}
  (H/h)^4 \ll H^{-4} \quad \text{if $H^2\ll h$},
\end{equation*}
 which again indicates that
\begin{equation}\label{eq:TLLarge2}
  \kappa(\BTL\tA_h)<\kappa(\BOL\tA_h) \quad \text{for small $H$ (or equivalently
 for large $J$)  if $h$ is kept fixed.}
 \end{equation}
%
\section{Numerical  Results}\label{sec:Numerics}
  We consider the obstacle problem in \cite[Example~4.2]{BDS:2014:GFEM},
   where $\O=(-0.5,0.5)^2$,  $f=0$ and $\psi(x)=1-5|x|^2+|x|^4$.
    We discretize \eqref{eq:OP} by the PUM
 using rectangular patches (cf. Figure~\ref{Fig:PTU})
  with $h\approx 2^{-\ell}$, where $\ell$ is the refinement level.
  As $\ell$ increases from 1 to 8, the number of degrees of freedom
  increases from 4 to 583696.  The discrete variational inequalities are
 solved by the primal-dual active set (PDAS) algorithm in Section~\ref{sec:PDAS}.
\par
 For the purpose of comparison,
  we first solve the auxiliary systems in each iteration of the
  PDAS algorithm by the conjugate gradient (CG) method without a preconditioner.  The
  average condition number during the PDAS iteration and the time to solve
  the variational inequality are presented
  in Table~\ref{table:CG}.  The PDAS iterations fail to stop (DNC) within 48 hours at  level
 $8$.
\begin{table}[!htb]
\begin{center}
\caption{Average condition number ($\kappa$) and time to solve the variational inequality
($t_\text{solve}$) by the CG algorithm}
\label{table:CG}
\begin{tabular}{|c|c|c|}
 \hline &&\\[-13pt]
$\ell$ & $\kappa$ & $t_\text{solve}$ (sec)\\
\hline &&\\[-12pt]
1 & $1.0000\times10^0$ &  $8.4889\times10^{-2}$ \\
\hline &&\\[-12pt]
2 & $2.8251\times 10^2$  &  $1.1865\times10^{-1}$       \\
\hline &&\\[-12pt]
3  & $6.2071\times 10^3$ &  $8.7772\times10^{-1}$          \\
\hline &&\\[-12pt]
4  & $9.3827\times10^4$ &   $9.7040\times10^{+0}$        \\
\hline &&\\[-12pt]
5   & $1.7843\times10^6$ &   $1.1611\times10^{+2}$        \\
\hline  &&\\[-12pt]
6   &$3.0294\times10^7$ &    $4.3516\times10^{+3}$      \\
\hline &&\\[-12pt]
7   &$4.9776\times10^8$ &   $9.8090\times10^{+4} $     \\
\hline &&\\[-12pt]
8  &  $8.0687\times10^9$ &   DNC        \\
\hline
\end{tabular}
\end{center}
\end{table}
\par
  We then solve the auxiliary systems  by the preconditioned conjugate gradient (PCG) method, using the
  additive Schwarz preconditioners associated with
  $J$ subdomains.  The mesh size $H$ for the coarse generalized finite element space is $\approx 1/\sqrt J$.
    We say the PCG method
   has converged if $\|B r \|_2 \leq 10^{-15} \|b\|_2,$ where $B$ is the preconditioner,
    $r$ is the residual, and $b$ is the load vector.
   The initial guess for the PDAS algorithm is taken to be the solution at the
   previous level, or 0 if $2^{2\ell}=J$.  The subdomain problems and coarse problem are solved by a direct
   method based on the Cholesky factorization.
\subsection{Small Overlap}\label{subsec:SmallNumerics}
 In this case we have $\delta\approx h$.
 The numbers of PDAS iterations for the one-level and two-level algorithms are given in
 Table~\ref{table:PDASSmall}.  The results are similar.  (The numbers only differ at three locations
 where they appear in red.)
 For both algorithms, the PDAS iterations fail to stop  within 48 hours at  level
 $8$  when $J=4$, which is due to the large sizes of the subdomain problems.

\begin{table}[!htb]
\begin{center}
\caption{Number of PDAS iterations with small overlap}
\label{table:PDASSmall}
\begin{tabular}{|c||c|c||c|c||c|c||c|c||}
\hline
&\multicolumn{2}{|c||}{}&\multicolumn{2}{c||}{}&\multicolumn{2}{c||}{}&\multicolumn{2}{c||}{}\\[-13pt]
&\multicolumn{2}{|c||}{$J=4$}&\multicolumn{2}{c||}{$J=16$}&\multicolumn{2}{c||}{$J=64$}
&\multicolumn{2}{c||}{$J=256$} \\
\cline{2-9}
$\ell$&\tiny{one-level} & \tiny{two-level}&\tiny{one-level} & \tiny{two-level}&\tiny{one-level}
 & \tiny{two-level}
&\tiny{one-level} & \tiny{two-level}\\
\hline&&&&&&&&\\[-13pt]
1 & 1 & 1 & - &- &- &- &- &-\\
\hline&&&&&&&&\\[-13pt]
2 & 5& 5 & 4& 4 &- &- &- &-\\
\hline&&&&&&&&\\[-13pt]
3& 12 & 12 & 12& 12 &14 &14 & - &-\\
\hline&&&&&&&&\\[-13pt]
4& 21 &21&21 &21 &{\red 21}&{\red 30} & 29 & 29  \\
\hline&&&&&&&&\\[-13pt]
5 & 22&22  &22&22  &22  &22  &{\red 22} &{\red 47}\\
\hline&&&&&&&&\\[-13pt]
6 &47  &47  &47  &47 &47  &47 &{\red  47} &{\red 89}\\
\hline &&&&&&&&\\[-13pt]
7 & 66  &66  &66 &66 &66 &66 &66 & 66  \\
\hline &&&&&&&&\\[-13pt]
8 & \text{DNC}& \text{DNC} &64 &64 &64 &64 &64 &64\\
\hline
\end{tabular}
\end{center}
\end{table}
\par
 The average condition number of the preconditioned auxiliary
 systems during the PDAS iterations
 are reported in Tables~\ref{table:CNOLSmall} and \ref{table:CNTLSmall}.
 Comparing to the average condition number in Table~\ref{table:CG}, both algorithms
  show marked improvement.
 The behavior of the
 condition numbers for the one-level algorithm in Table~\ref{table:CNOLSmall} agrees with the observations
 in \eqref{eq:OLSmall2} and \eqref{eq:OLSmall3}.  A comparison of Table~\ref{table:CNOLSmall} and
 Table~\ref{table:CNTLSmall} shows that
\begin{equation*}
  \max \frac{\kappa(\BTL\tA_h)}{\kappa(\BOL\tA_h)}\approx 1.24,
\end{equation*}
 where the maximum is taken over all the corresponding entries in
 Table~\ref{table:CNOLSmall} and
 Table~\ref{table:CNTLSmall}, which agrees with \eqref{eq:TwoBetter}.
 Moreover, $\kappa(\BTL\tA_h)$ is
 smaller than $\kappa(\BOL\tA_h)$ for $J$ large, as observed in \eqref{eq:TLSmall1}.
 \begin{table}[!htb]
\begin{center}
\caption{Average condition number, one-level with small overlap}
\label{table:CNOLSmall}
\begin{tabular}{|c|c|c|c|c|}
\hline &&&&\\[-12pt]
$\ell$ & $J=4$ & $J=16$ & $J=64$ & $J=256$ \\
\hline &&&&\\[-12pt]
1 & $0.000\times 10^{+0}$ & - & - & - \\
\hline &&&&\\[-12pt]
2 & $3.950\times 10^{+0}$ & $2.187\times 10^{+0}$ & - & - \\
\hline &&&&\\[-12pt]
3 & $4.351\times 10^{+0}$ & $6.395\times 10^{+0}$ & $5.886\times 10^{+0}$ & - \\
\hline &&&&\\[-12pt]
4 & $4.928\times 10^{+0}$ & $1.116\times 10^{+1}$ & $2.301\times 10^{+1}$  & $3.751\times 10^{+1}$ \\
\hline &&&&\\[-12pt]
5 & $9.825\times 10^{+0}$ & $6.154\times 10^{+1}$ & $1.057\times 10^{+2}$  & $2.846\times 10^{+2}$  \\
\hline &&&&\\[-12pt]
6 & $2.489\times 10^{+1}$ & $4.296\times 10^{+2}$ & $8.012\times 10^{+2}$  & $1.504\times 10^{+3}$ \\
\hline &&&&\\[-12pt]
7 & $1.441\times 10^{+2}$ & $3.341\times 10^{+3}$  & $6.397\times 10^{+3}$ & $1.226\times 10^{+4}$  \\
\hline &&&&\\[-12pt]
8 & $1.053\times 10^{+3}$ & $2.650\times 10^{+4}$  & $5.135\times 10^{+4}$ & $9.913\times 10^{+4}$ \\
\hline
\end{tabular}
\end{center}
\end{table}
\begin{table}[!htb]
\begin{center}
\caption{Average condition number, two-level with small overlap}
\label{table:CNTLSmall}
\begin{tabular}{|c|c|c|c|c|}
\hline &&&&\\[-12pt]
$\ell$ & $J=4$ & $J=16$ & $J=64$ & $J=256$ \\
\hline &&&&\\[-12pt]
1 & $0.000\times 10^{+0}$ & - & - & - \\
\hline &&&&\\[-12pt]
2 & $4.909\times 10^{+0}$ & $2.486\times 10^{+0}$ & - & - \\
\hline &&&&\\[-12pt]
3 & $5.161\times 10^{+0}$ & $6.296\times 10^{+0}$ & $6.219\times 10^{+0}$ & - \\
\hline &&&&\\[-12pt]
4 & $5.568\times 10^{+0}$ & $1.235\times 10^{+1}$ & $9.804\times 10^{+0}$ & $3.880\times 10^{+1}$ \\
\hline &&&&\\[-12pt]
5 & $1.025\times 10^{+1}$ & $5.147\times 10^{+1}$ & $2.704\times 10^{+1}$ & $1.097\times 10^{+1}$  \\
\hline &&&&\\[-12pt]
6 & $2.519\times 10^{+1}$ & $3.268\times 10^{+2}$ & $6.817\times 10^{+1}$ & $3.508\times 10^{+1}$ \\
\hline &&&&\\[-12pt]
7 & $1.447\times 10^{+2}$ & $1.164\times 10^{+3}$ & $3.056\times 10^{+2}$ & $7.647\times 10^{+1}$\\
\hline &&&&\\[-12pt]
8 & $1.060\times 10^{+3}$ & $8.726\times 10^{+3}$ & $2.034\times 10^{+3}$ & $3.401\times 10^{+2}$ \\
\hline
\end{tabular}
\end{center}
\end{table}
\par
 The time to solve for both algorithms is documented in Tables~\ref{table:TimeOLSmall} and
 \ref{table:TimeTLSmall}. To compare the performance of these two algorithms, we have recorded in red
  the
 faster times that appear in Table~\ref{table:TimeTLSmall}.
 It is observed that the two-level algorithm is advantageous when
 $h$ is small and $J$ is large, which agrees with the observation in \eqref{eq:TLSmall1}.
\par
 Comparing to the solution time in Table~\ref{table:CG}, we see that the PCG using either preconditioner
 is much more efficient for the large problems at higher refinement levels.  At refinement level 7,
  the solution time for the one-level algorithm using 256 subdomains is roughly 100 times faster
  than that for the CG algorithm without a preconditioner, and the solution time for the
   two-level algorithm using 256 subdomains is
  roughly 200 times faster.
\begin{table}[!htb]
\begin{center}
\caption{Time to solve (in seconds), one-level with small overlap}
\label{table:TimeOLSmall}
\begin{tabular}{|c|c|c|c|c|}
\hline&&&&\\[-12pt]
$\ell$ & $J=4$ & $J=16$ & $J=64$ & $J=256$ \\
\hline&&&&\\[-12pt]
1 & $8.1317\times 10^{-2}$ & - & - & - \\
\hline&&&&\\[-12pt]
2 & $3.8073\times 10^{-1}$ & $1.0504\times 10^{+0}$ & - & -  \\
\hline&&&&\\[-12pt]
3 & $1.2931\times 10^{+0}$ & $5.7371\times 10^{+0}$ & $9.3857\times 10^{+0}$ & -  \\
\hline&&&&\\[-12pt]
4 & $4.1499\times 10^{+0}$ & $1.2460\times 10^{+1}$ & $2.2179\times 10^{+1}$ &  $4.2931\times 10^{+1}$  \\
\hline&&&&\\[-12pt]
5 & $2.7210\times 10^{+1}$ & $2.5699\times 10^{+1}$ & $3.0170\times 10^{+1}$ &  $6.0450\times 10^{+1}$ \\
\hline&&&&\\[-12pt]
6 & $6.9396\times 10^{+2}$ & $2.3698\times 10^{+2}$ & $1.5836\times 10^{+2}$ & $1.9689\times 10^{+2}$ \\
\hline&&&&\\[-12pt]
7 & $1.4585\times 10^{+4}$ & $3.2359\times 10^{+3}$ & $9.9417\times 10^{+2}$ & $8.4106\times 10^{+2}$ \\
\hline&&&&\\[-12pt]
8 & DNC & $4.0802\times 10^{+4}$ & $9.2843\times 10^{+3}$ & $3.9043\times 10^{+3}$ \\
\hline
\end{tabular}
\end{center}
\end{table}
\begin{table}[!htb]
\begin{center}
\caption{Time to solve (in seconds), two-level with small overlap, where the  entries in red
 represent faster times than those in Table~\ref{table:TimeOLSmall} (one-level)}
\label{table:TimeTLSmall}
\begin{tabular}{|c|c|c|c|c|}
\hline&&&&\\[-12pt]
$\ell$ & $J=4$ & $J=16$ & $J=64$ & $J=256$ \\
\hline&&&&\\[-12pt]
1 & $8.5648\times 10^{-2}$ & - & - & - \\
\hline&&&&\\[-12pt]
2 & $5.0521\times 10^{-1}$ & $1.3693\times 10^{+0}$ & - & - \\
\hline&&&&\\[-12pt]
3 & $1.8366\times 10^{+0}$ & $7.8522\times 10^{+0}$ & $1.1167\times 10^{+1}$ & - \\
\hline&&&&\\[-12pt]
4 & $5.0832\times 10^{+0}$ & $1.7047\times 10^{+1}$ & $3.9022\times 10^{+1}$ & $6.7449\times 10^{+1}$ \\
\hline&&&&\\[-12pt]
5 & $2.8294\times 10^{+1}$ & $3.2915\times 10^{+1}$ & $4.1276\times 10^{+1}$ & $1.9943\times 10^{+2}$ \\
\hline&&&&\\[-12pt]
6 & $6.9796\times 10^{+2}$ & $2.6202\times 10^{+2}$ & $1.6060\times 10^{+2}$ & $4.0555\times 10^{+2}$ \\
\hline&&&&\\[-12pt]
7 &{\red $1.4319\times 10^{+4}$} &{\red $3.0723\times 10^{+3}$} &
 {\red$7.8729\times 10^{+2}$} &{\red $4.3991\times 10^{+2}$}\\
\hline&&&&\\[-12pt]
8 & DNC &{\red $3.9900\times 10^{+4}$} &{\red $6.3162\times 10^{+3}$}
 &{\red $1.1682\times 10^{+3}$} \\
\hline
\end{tabular}
\end{center}
\end{table}
\par
 The averaged condition number for the PDAS
 iteration at refinement level 8  together with the time to solve the variational inequality
 are displayed in Table~\ref{table:Scaling} with an increasing number of subdomains.
 The scalability of the algorithm is clearly observed.
\begin{table}[!htb]
\begin{center}
\caption{Average condition number ($\kappa$) and time to solve the
 variational inequality ($t_{\text{solve}}$)
 for the two-level algorithm with small overlap
 at refinement level 8}
\label{table:Scaling}
\begin{tabular}{|c|c|c|}
\hline&&\\[-13pt]
$J$ & $\kappa$ & $t_\text{solve}$ (sec) \\
\hline &&\\[-12pt]
4 & $1.060\times 10^{+3}$ & DNC  \\
\hline &&\\[-12pt]
16  & $8.726\times 10^{+3}$ & 4.4624$\times 10^{+4}$  \\
\hline &&\\[-12pt]
64 & $2.034\times 10^{+3}$ & 5.3898$\times 10^{+3}$  \\
\hline &&\\[-12pt]
256 & $3.401\times 10^{+2}$ & 1.0143$\times 10^{+3}$  \\ \hline
\end{tabular}
\end{center}
\end{table}
\subsection{Generous Overlap}\label{subsec:LargeNumerics}
 In this case we have $\delta\approx H$.
The numbers of PDAS iterations for the one-level and two-level algorithms are given in
 Table~\ref{table:PDASGenerous}.  The results are  again similar.  (The numbers only
 differ at one location where they appear in red.)
 For both algorithms, the PDAS iterations fail to stop within 48 hours at  level
 $7$ when we only use 4 subdomains, and at level $8$ when we only use up to $16$ subdomains.
 Comparing with Table~\ref{table:PDASSmall}, we clearly see the
 adverse effect of
 the large overlap on the sizes of the subdomain problems and on the communication time.
\begin{table}[!htb]
\begin{center}
\caption{Number of PDAS iterations with generous overlap}
\label{table:PDASGenerous}
\begin{tabular}{|c||c|c||c|c||c|c||c|c||}
\hline
&\multicolumn{2}{|c||}{}&\multicolumn{2}{c||}{}&\multicolumn{2}{c||}{}&\multicolumn{2}{c||}{}\\[-13pt]
&\multicolumn{2}{|c||}{$J=4$}&\multicolumn{2}{c||}{$J=16$}&\multicolumn{2}{c||}{$J=64$}
&\multicolumn{2}{c||}{$J=256$} \\
\cline{2-9}
$\ell$&\tiny{one-level} & \tiny{two-level}&\tiny{one-level} & \tiny{two-level}&\tiny{one-level}
 & \tiny{two-level}
&\tiny{one-level} & \tiny{two-level}\\
\hline&&&&&&&&\\[-13pt]
1&1&1&-&-&-&-&-&-\\
\hline&&&&&&&&\\[-13pt]
2&5&5&4&4&-&-&-&-\\
\hline&&&&&&&&\\[-13pt]
3&12&12&12&12&14&14&-&-\\
\hline&&&&&&&&\\[-13pt]
4&21&21&21&21&{\red 21}&{\red 30}&29&29\\
\hline&&&&&&&&\\[-13pt]
5&22&22&22&22&22&22&22&22\\
\hline&&&&&&&&\\[-13pt]
6&47&47&47&47&47&47&47&47\\
\hline&&&&&&&&\\[-13pt]
7&\text{DNC}&\text{DNC}&66&66&66&66&66&66\\
\hline&&&&&&&&\\[-13pt]
8&\text{DNC}&\text{DNC}&\text{DNC}&\text{DNC}&64&64&64&64\\
\hline
\end{tabular}
\end{center}
\end{table}
\par
 The average condition numbers of the preconditioned auxiliary systems observed during the PDAS iterations
 are displayed in Tables~\ref{table:CNOLGenerous} and \ref{table:CNTLGenerous}.
 At refinement level 8, the average condition numbers for the one-level preconditioner are less than 52
 and those for the two-level preconditioner are less than 16, a big improvement over
  the average condition number of
 $8\times10^9$ for the auxiliary system itself.
\begin{table}[!htb]
\begin{center}
\caption{Average condition number, one-level with generous overlap}
\label{table:CNOLGenerous}
\begin{tabular}{|c|c|c|c|c|}
\hline &&&&\\[-12pt]
$\ell$ & $J=4$ & $J=16$ & $J=64$ & $J=256$ \\
\hline &&&&\\[-12pt]
1 & $0.000\times 10^{+0}$ & - & - & - \\
\hline &&&&\\[-12pt]
2 & $1.000\times 10^{+0}$ & $2.187\times 10^{+0}$ & - & - \\
\hline &&&&\\[-12pt]
3 & $1.000\times 10^{+0}$ & $2.929\times 10^{+0}$  & $5.886\times 10^{+0}$  & - \\
\hline &&&&\\[-12pt]
4 & $1.000\times 10^{+0}$ & $2.695\times 10^{+0}$  & $6.083\times 10^{+0}$  & $3.751\times 10^{+1}$ \\
\hline &&&&\\[-12pt]
5 & $1.000\times 10^{+0}$ & $2.712\times 10^{+0}$  & $6.129\times 10^{+0}$ & $4.914\times 10^{+1}$ \\
\hline &&&&\\[-12pt]
6 & $1.000\times 10^{+0}$ & $2.693\times 10^{+0}$ & $6.216\times 10^{+0}$  & $5.020\times 10^{+1}$ \\
\hline &&&&\\[-12pt]
7 & \text{DNC} & $2.669\times 10^{+0}$ & $6.289\times 10^{+0}$ & $5.145\times 10^{+1}$\\
\hline &&&&\\[-12pt]
8 & \text{DNC} & \text{DNC} & $6.316\times 10^{+0}$  & $5.207\times 10^{+1}$ \\
\hline
\end{tabular}
\end{center}
\end{table}
\begin{table}[!htb]
\begin{center}
\caption{Average condition number, two-level with generous overlap}
\label{table:CNTLGenerous}
\begin{tabular}{|c|c|c|c|c|}
\hline &&&&\\[-12pt]
$\ell$ & $J=4$ & $J=16$ & $J=64$ & $J=256$ \\
\hline &&&&\\[-12pt]
1 & $0.000\times 10^{+0}$ & - & - & - \\
\hline &&&&\\[-12pt]
2 & $1.250\times 10^{+0}$ & $2.486\times 10^{+0}$ & - & - \\
\hline &&&&\\[-12pt]
3 & $1.250\times 10^{+0}$ & $3.012\times 10^{+0}$ & $6.219\times 10^{+0}$ & - \\
\hline &&&&\\[-12pt]
4 & $1.250\times 10^{+0}$ & $2.785\times 10^{+0}$ & $4.386\times 10^{+0}$ & $3.880\times 10^{+1}$ \\
\hline &&&&\\[-12pt]
5 & $1.250\times 10^{+0}$ & $2.729\times 10^{+0}$ & $5.213\times 10^{+0}$ & $1.164\times 10^{+1}$ \\
\hline &&&&\\[-12pt]
6 & $1.250\times 10^{+0}$ & $2.696\times 10^{+0}$ & $5.310\times 10^{+0}$ & $1.342\times 10^{+1}$ \\
\hline &&&&\\[-12pt]
7 & \text{DNC} & $2.669\times 10^{+0}$ & $5.653\times 10^{+0}$ & $1.489\times 10^{+1}$\\
\hline &&&&\\[-12pt]
8 & \text{DNC} & \text{DNC} & $5.748\times 10^{+0}$ & $1.663\times 10^{+1}$ \\
\hline
\end{tabular}
\end{center}
\end{table}
\par
 The behavior of the
 condition numbers for the one-level algorithm in
 Table~\ref{table:CNOLGenerous} agrees with the observations
 in \eqref{eq:OLLarge2} and \eqref{eq:OLLarge3}.
 The behavior of condition numbers for the two-level algorithm in Table~\ref{table:CNTLGenerous}
 also agrees with the observation in \eqref{eq:TLLarge1}.
   A comparison of Table~\ref{table:CNOLGenerous} and
 Table~\ref{table:CNTLGenerous} indicates that $\kappa(\BTL\tA_h)$ is
 smaller than $\kappa(\BOL\tA_h)$ for $J$ large, as observed in
 \eqref{eq:TLLarge2}.
  Moreover, we have
\begin{equation*}
  \max\frac{\kappa(\BTL\tA_h)}{\kappa(\BOL\tA_h)}\approx 1.25,
\end{equation*}
 where the maximum is taken over all the corresponding entries in
 Table~\ref{table:CNOLGenerous} and
 Table~\ref{table:CNTLGenerous},
 which agrees with \eqref{eq:TwoBetter}.
\par
 The time to solve for both algorithms is presented in Tables~\ref{table:TimeOLGenerous} and
 \ref{table:TimeTLGenerous}.  A comparison of these two tables  again indicates that the two-level
 algorithm is only advantageous when $h$ is small and $J$ is large.  For $J=64$, this is observed
 for level $7$ and $8$.  For $J=256$, this is not yet observed at level $8$.
\par
 We also compare Table~\ref{table:TimeOLSmall} (resp., Table~\ref{table:TimeTLSmall})
 and Table~\ref{table:TimeOLGenerous} (resp.,Table~\ref{table:TimeTLGenerous}) by recording the
 faster times that appear in  Table~\ref{table:TimeOLGenerous}
 (resp., Table~\ref{table:TimeTLGenerous}) in an enlarged format.  It is observed that, for
 a fixed number of subdomains,
 the algorithm with generous overlap eventually loses its advantage
 as the one with a better condition number because of the increase of communication time when the
 mesh is refined.
\begin{table}[!htb]
\begin{center}
\caption{Time to solve (in seconds), one-level with generous overlap, where
 the entries in red represent
 faster times in comparison with those in Table~\ref{table:TimeOLSmall} (small overlap)}
\label{table:TimeOLGenerous}
\begin{tabular}{|c|c|c|c|c|}
\hline&&&&\\[-12pt]
$\ell$ & $J=4$ & $J=16$ & $J=64$ & $J=256$ \\
\hline&&&&\\[-12pt]
1 & $1.0883\times 10^{-1}$ & - & - & - \\
\hline&&&&\\[-12pt]
2 &{\red $3.5694\times 10^{-1}$} & $1.0510\times 10^{+0}$ & - & - \\
\hline&&&&\\[-12pt]
3 &{\red $1.0364\times 10^{+0}$} & $6.0472\times 10^{+0}$ &{\red$9.1838\times 10^{+0}$} & - \\
\hline&&&&\\[-12pt]
4 & $6.9156\times 10^{+0}$ & $1.4521\times 10^{+1}$ &{\red $1.8781\times 10^{+1}$}
 &{\red $4.1750\times 10^{+1}$} \\
\hline&&&&\\[-12pt]
5 & $1.1226\times 10^{+2}$ & $6.4720\times 10^{+1}$ &{\red $2.8115\times 10^{+1}$}
 &{\red$4.3907\times 10^{+1}$} \\
\hline&&&&\\[-12pt]
6 & $4.2799\times 10^{+3}$ & $1.5403\times 10^{+3}$ & $2.2077\times 10^{+2}$
&{\red $1.2364\times 10^{+2}$}\\
\hline&&&&\\[-12pt]
7 & DNC & $3.2500\times 10^{+4}$ & $3.5810\times 10^{+3}$
 &{\red $5.1525\times 10^{+2}$}\\
\hline&&&&\\[-12pt]
8 & DNC & DNC & $5.6323\times 10^{+4}$ & $3.9372\times 10^{+3}$ \\
\hline
\end{tabular}
\end{center}
\end{table}
\begin{table}[!htb]
\begin{center}
\caption{Time to solve (in seconds), two-level with generous overlap, where the
 entries in red represent
 faster times in comparison with those in Table~\ref{table:TimeTLSmall} (small overlap)}
\label{table:TimeTLGenerous}
\begin{tabular}{|c|c|c|c|c|}
\hline &&&&\\[-12pt]
$\ell$ & $J=4$ & $J=16$ & $J=64$ & $J=256$ \\
\hline &&&&\\[-12pt]
1 & {\red$7.7797\times 10^{-2}$} & - & - & - \\
\hline &&&&\\[-12pt]
2 &{\red$3.9949\times 10^{-1}$} &{\red$1.3328\times 10^{+0}$} & - & - \\
\hline &&&&\\[-12pt]
3 &{\red $1.1620\times 10^{+0}$} &{\red $7.1329\times 10^{+0}$} &
{\red $1.0951\times 10^{+1}$} & - \\
\hline &&&&\\[-12pt]
4 & $7.4408\times 10^{+0}$ & $1.7123\times 10^{+1}$ &
{\red$3.2808\times 10^{+1}$} & {\red$5.7805\times 10^{+1}$} \\
\hline &&&&\\[-12pt]
5 & $1.1291\times 10^{+2}$ & $6.7401\times 10^{+1}$ &{\red$3.3208\times 10^{+1}$} &
{\red $8.2053\times 10^{+1}$} \\
\hline &&&&\\[-12pt]
6 & $4.0284\times 10^{+3}$ & $1.5499\times 10^{+3}$ & $2.2770\times 10^{+2}$ &
{\red$1.4431\times 10^{+2}$} \\
\hline &&&&\\[-12pt]
7 & DNC & $3.2466\times 10^{+4}$ & $2.9299\times 10^{+3}$ & $5.2275\times 10^{+2}$\\
\hline &&&&\\[-12pt]
8 & DNC & DNC & $4.0677\times 10^{+4}$ & $4.7345\times 10^{+3}$ \\
\hline
\end{tabular}
\end{center}
\end{table}

%
%
\section{Concluding Remarks}\label{sec:Conclusions}
 We investigated two additive Schwarz domain decomposition preconditioners for the auxiliary systems that
  appear in a primal-dual active algorithm for the  numerical solution of
  the obstacle problem for the clamped Kirchhoff plate, where the discretization is based on a
  partition of unity generalized finite element method.
 \par
  The condition number estimates for the one-level
  additive Schwarz preconditioner are identical to those for the plate bending problem in the
  absence of an obstacle.  On the other hand, the
   condition number estimates for the two-level additive Schwarz preconditioner are
    different because the creation of the coarse problem  requires a truncation procedure at the fine level.
 \par
  The theoretical estimates are confirmed by numerical results, which also
   indicate that for large problems the best performance (in terms of time to solve)
   is obtained by the two-level algorithm with small overlap.
 \par
  In our computations we solve the subdomain problems and the coarse problem using
  a direct solve based on the Cholesky factorizations of the  matrices.  Because the active set and hence the matrices change from one PDAS iteration to the next,
    we have to recompute the Cholesky factorization during each PDAS factorization, which is time consuming
  for large matrices.  Since the change in the active set
  eventually becomes small, the performance of our method can be improved by using
  a fast modification of the Cholesky factorization
  that is discussed for example in \cite{DH:1999:Cholesky,DH:2001:Cholesky,DH:2005:Cholesky}.
%

%
\end{document}